\documentclass[graybox]{svmult}

\usepackage{mathptmx}       % selects Times Roman as basic font
\usepackage{helvet}         % selects Helvetica as sans-serif font
\usepackage{courier}        % selects Courier as typewriter font
\usepackage{makeidx}         % allows index generation
\usepackage{graphicx}        % standard LaTeX graphics tool
\usepackage{multicol}        % used for the two-column index
\usepackage[bottom]{footmisc}% places footnotes at page bottom

\usepackage{amsmath,amssymb}
\let\mathcal\mathscr

\def\Z{{\bf Z}}
\def\N{{\bf N}}

\def\C{{\bf C}}

\def\R{{\bf R}}

\def\P{{\bf P}}

\def\phi{{\varphi}}

\def\epsilon{{\varepsilon}}

\def\dra{\dashrightarrow}

\def\isom{\simeq}

\def\ie{\hbox{i.e.}}

\DeclareMathOperator{\codim}{codim}
\DeclareMathOperator{\conv}{conv}
\DeclareMathOperator{\GL}{GL}
\DeclareMathOperator{\MV}{MV}
\DeclareMathOperator{\vol}{vol}

\begin{document}

\title*{Monomial transformations of the projective space}
\author{Olivier Debarre and Bodo Lass}
\institute{Olivier Debarre \at D\'epartement de  Math\'ematiques et Applications -- CNRS UMR 8553,
\'Ecole normale sup\'erieure,
45 rue d'Ulm, 75230 Paris Cedex 05, France, \email{olivier.debarre@ens.fr}
\and Bodo Lass \at Institut Camille Jordan -- CNRS UMR 5208, Universit\'e Claude Bernard -- Lyon 1, 69622 Villeurbanne Cedex,
France, \email{lass@math.univ-lyon1.fr}}

\maketitle

\abstract{We prove that, over any field, the dimension of the indeterminacy locus of a rational map $f:\P^n\dra \P^n$ defined by monomials of the same degree $d$ with no common factors is at least $(n-2)/2$, provided that the degree of $f$ as a map is not divisible by $d$.
 This implies upper bounds on the multidegree of $f$ and in particular, when $f$ is birational, on the degree of $f^{-1}$.
}

\section{Introduction}
\label{sec:0}
We denote by $\P^n$ the $n$-dimensional projective space over a fixed field.
 A monomial transformation of $\P^n $ is a   rational map
$f:\P^n\dra \P^n$ whose components $  f_0,\dots,f_n $ are monomials (of the same positive degree $d(f)$ and with no common factors) in the variables $x_0,\dots,x_n$. 

Monomial transformations are of course very special among all rational transformations, but also much easier to study. For this reason,   they   have recently attracted some attention. In particular, there is a description of all {\em birational} 
monomial transformations $f$ with $d(f)=2$  in \cite{cs}, \S2, from which it follows   that   $d(f^{-1})$ is then at most equal to $n$ (\cite{cs}, Theorem 2.6). Extensive computer calculations were then performed by Johnson in \cite{joh}
%for   monomial birational transformations $f$ (and small $d(f)$ and $n$) 
and led him to suggest that the largest possible value for  $d(f^{-1})$ should be
$\frac{(d(f)-1)^n-1}{d(f)-2}$ when $d(f)\ge 3$.

These values should be compared with the optimal bound $d(g^{-1})\le d(g)^{n-1}$ for all birational  transformations $g$ of $\P^n$.
This maximal value for $d(g^{-1})$ is attained if and only if the indeterminacy locus of $g$ is finite (see \S\ref{sec:3}), hence one is led to think that the indeterminacy locus of a monomial map should be rather large. This is what we prove in Theorem \ref{main}: {\em the dimension of the  indeterminacy locus of a monomial map $f:\P^n\dra \P^n$   is at least $(n-2)/2$, provided that the degree of $f$ as a map is not divisible by $d(f)$.}

We show in \S\ref{sec:3} that this implies a bound on $d(f^{-1})$ for all birational 
monomial transformations $f$, which is however not as good as the one suggested by Johnson.

 \section{Monomial   transformations}
\label{sec:1}

We represent  a
monomial transformation  $f:\P^n\dra \P^n$, with components $  f_0,\dots,f_n $,
by the $(n+1)\times(n+1)$ matrix $A=(a_{ij})_{0\le i,j\le n}$ whose $i$th row lists the exponents of $f_i$. With this notation, one has $f_A\circ f_B=f_{AB}$. 

The following proposition is   elementary (\cite{gsp}; \cite{sv1} Lemma 1.2; \cite{joh}).

\begin{proposition}\label{p1}
With the notation above, we have
$$|\det(A)|=d(f) \deg(f).$$
In particular, $f$ is birational if and only if $ |\det(A)|=d(f)$.
\end{proposition}

\begin{proof}
The condition that all monomials $  f_0,\dots,f_n $ have the same degree $d:=d(f)$ means that in each row of $A$, the sum of the entries is $d$.
 Adding all columns to the 0th column, then subtracting the first row from all other rows we obtain
$$\det(A)=\left|
\begin{array}{cccc}
 d&a_{01}&\cdots&a_{0n}\\
\vdots&\vdots&&\vdots\\
d&a_{n1}&\cdots&a_{nn} 
\end{array}
\right|=d\, \left| 
\begin{array}{cccc}
1&a_{01}&\cdots&a_{0n}\\
\vdots&\vdots&&\vdots\\
1&a_{n1}&\cdots&a_{nn} 
\end{array}
\right|=d\,\det(M),
$$
where $M:=(m_{ij})_{1\le i,j\le n}$ is defined by $m_{ij}:=a_{ij}-a_{0j}$. If $T\isom (\C^*)^n\subset \P^n$ is 
 the torus  defined by $x_0\cdots x_n\ne 0$, the map $f$ induces a morphism $f_T:T\to T$
 given by
$$f_T(x_1,\dots,x_n)=(x_1^{m_{11}}\cdots x_n^{m_{1n}},\dots ,x_1^{m_{n1}}\cdots x_n^{m_{nn}}).$$
 The induced map $\widehat f_T:\widehat T\to \widehat T$ between algebraic character groups (where $\widehat T$ is the free abelian group $\Z^n$) is given by the transposed matrix $M^T:\Z^n\to \Z^n$. Performing elementary operations on $M$ amounts to composing $f_T$ with monomial  automorphisms, so we can reduce to the case where $M$ is diagonal, in which case it is obvious that the degree of the morphism $f_T$ (which is the same as the degree of $f$) is $|\det(M)|$.
 \qed
\end{proof}

\begin{corollary}
With the notation above,   $f$ is birational if and only if $ |\det(A)|=d(f)$. Its inverse is then also a monomial transformation.
\end{corollary}

\begin{proof}
It is clear from the proof above that  $f$ is birational if and only if $ |\det(M)|=1$, \ie, if and only if $M\in \GL_n(\Z)$ or, equivalently, if and only if $f_T$ is an isomorphism, whose inverse is then given by the matrix $M^{-1}$. It is therefore a monomial transformation.
 \qed
\end{proof}

\section{Indeterminacy locus}
\label{sec:2}

 In this section, $f:\P^n \dra \P^n $ is again a monomial transformation. {\em We assume that its components $  f_0,\dots,f_n $ have no common factors.} In terms of the matrix $A$ defined in \S\ref{sec:1}, this means that each column of $A$ has at least one 0 entry.

 The indeterminacy locus $B$ of $f$ is then the subscheme of $\P^n$ defined by the equations $f_0,\dots,f_n$. Its blow-up $\widehat X\to X$ is the graph $\Gamma_f\to \P^n$ of $f$ (\cite{dol}, \S1.4).
 
 For each nonempty subset $J\subsetneq\{0,\dots,n\}$ such that the $(n+1)\times |J|$ matrix $A_J$ constructed from the columns of $A$ corresponding to the elements of $J$ has no zero rows, we obtain a linear space contained in $B$   by setting $x_j=0$ for all $j\in J$. Its codimension in $\P^n$ is $|J|$. Moreover, $B_{\rm red}$ is the union of all such linear spaces. 
 
 \begin{theorem}\label{main}
 Let $f:\P^n \dra \P^n $ be   a dominant   transformation  defined by monomials of degree $d$ with no common factors. If the degree of $f$ is not divisible by $d$, the dimension of the indeterminacy locus of $f$  is at least $(n-2)/2$.
 \end{theorem}
 
 The condition on the degree is necessary, as shown by the morphism $(x_0,\dots,x_n)\mapsto (x^d_0,\dots,x^d_n)$ of degree $d^n$ (whose indeterminacy locus is empty).
 
 \begin{proof}
Since the determinant of the matrix $A$ is nonzero we may assume, upon permuting its rows and columns,   
that we have $a_{ii}\ne 0$ for all $i\in \{0,\dots,n\}$.

We then define an oriented graph on the set of vertices $\{0,\dots,n\}$ by adding an oriented edge from  $i$ to $j$ whenever $a_{ij}\ne 0$.
We then say that a vertex $x$ is equivalent to a vertex $y$ if and only if there exists an oriented path from $x$ to $y$ and  an oriented path from $y$ to $x$. This defines a partition of the set
$\{0,\dots,n\}$ into equivalence classes
 (note that $x$ is equivalent to $x$ since $a_{xx}\ne 0$).

Say that an equivalence class $X$ is greater than or equal to an equivalence class $Y$ if there is an oriented path from an element of $X$ to an element of $Y$ (there exists then an oriented path from any element of $X$ to any element of $Y$). This defines a partial order on the set of equivalence classes.

Choose a class $X$    minimal for this order. 
%, we look at the rows and columns of $A$ corresponding to the elements of $X$. 
Entries of $A$ in a row corresponding to an element of $X$ which are not in a column corresponding to an element of $X$ are then 0 (otherwise,   at least one oriented edge should come out of a $X$ to element not in $X$, contradicting the minimality of $X$). It follows that the determinant of the submatrix $A_X$ of $A$ corresponding to rows and columns of $X$ divides the determinant of $A$. The sum of all entries in a row of $A_X$ is $d$ hence, by the same reasoning used in the proof of Proposition \ref{p1}, the determinant of $A_X$ is nonzero, divisible by $d$.

Because of the condition $d\nmid \deg(f)$ and Proposition \ref{p1}, the determinant of $A$ is not divisible by $d^2$. In particular, our partial order has a unique minimal element $X$. Without loss of generality, we may assume $0\in X$. Every other vertex then has an oriented path to 0. In particular, we may define an acyclic function
$$\phi: \{1,\dots,n\}\to \{0,1,\dots,n\}$$
such that $(x,f(x))$ is an oriented edge of our graph for all $x\in \{1,\dots,n\}$ (``acyclic'' means that for all $x\in \{1,\dots,n\}$, there exists $k>0$ such that $f^k(x)=0$).

 We keep only the $n$ edges of the type $(x,f(x))$; since $a_{xf(x)}\ne 0$, they correspond to $n$ nonzero entries, off the diagonal, in each row $1,\dots,n$.
  Since our new graph on $\{0,\dots,n\}$ has $n$ edges and no cycles, we may color its vertices in black or white in such a way that $x$ and $f(x)$ have different colors, for all $x\in\{0,\dots,n\}$.
 
 We select the vertices of the color which has been used less often (if both colors have been used the same number of times, we select the vertices with the same color as   0). If 0 is not selected, we add it to the selection. We end up with at most $(n+2)/2$ selected vertices which are all on one of our $n$ edges or the loop at 0.

Consider the  submatrix of $A$ formed by the $\le  (n+2)/2$  columns corresponding to the selected vertices. In each row, there is a nonzero entry: in the   row 0, because 0 was selected and $a_{00}\ne 0$; in any other row $x$ because either $x$ was selected and $a_{xx}\ne 0$, or $f(x)$ was selected and  $a_{xf(x)}\ne 0$.
This proves the theorem.
 \qed
\end{proof}

   \begin{example}
 The bound in the theorem is sharp: for $d\ge 2$, one easily checks that the indeterminacy locus of   the birational automorphism (\cite{joh},   Example 2)
  \begin{equation}\label{besta}
f_{n,d}: (x_0,\dots,x_n)\mapsto (x_0^d,x_0^{d-1}x_1 ,x_1^{d-1}x_2 ,\dots,x_{n-1}^{d-1}x_n)
\end{equation}
 of $\P^n$ has dimension exactly $\lceil (n-2)/2\rceil$. 
%The corresponding matrix is
%\begin{equation}\label{besta}
%A=\begin{pmatrix}
% d&0&\cdots&\cdots&0\\
%d-1&1&0&\cdots&0\\
%0&\ddots&\ddots&\ddots&\vdots\\
%\vdots&\ddots&d-1&1&0\\
%0&  \cdots&0&d-1&1
%\end{pmatrix}
%.
%\end{equation}
But there are many other examples of birational monomial automorphisms of $\P^n$  with indeterminacy locus of dimension exactly $\lceil (n-2)/2\rceil$, such as monomial maps defined by matrices 
$$A=\begin{pmatrix}
 d&0&\cdots&\cdots&\cdots&\cdots&\cdots&\cdots&0\\
d-1&1&0&\cdots&\cdots&\cdots&\cdots&\cdots&0\\
0&d-1&1&0&\cdots&\cdots&\cdots&\cdots&0\\
a_{30}&a_{31}&a_{32}&1&0&\cdots&\cdots&\cdots&0\\
0&0&0&d-1&1&0&\cdots&\cdots&0\\
a_{50}&a_{51}&a_{52}&a_{53}&a_{54}&1&0&\cdots&0\\
\vdots&&&&&\ddots&\ddots&\ddots&\vdots\\
\vdots&&&&&&\ddots&\ddots&0\\
&&&&&&&&1
\end{pmatrix}
$$
where, for each odd $i$, we choose $a_{i0}\ne 0$ and $\sum_{j=0}^{i-1}a_{ij}=d-1$. The (reduced) indeterminacy locus is then defined by the equations
$$x_0=x_1x_2=x_3x_4=\cdots=0.$$
It has dimension $n-1-\lfloor n/2\rfloor= \lceil (n-2)/2\rceil$.
%
%Since the first row of $A$ is $\begin{pmatrix}
% d&0&\cdots&0\end{pmatrix}$, the matrix $M$ from the proof of Proposition \ref{p1} is the matrix obtained from $A$ by deleting its first row and column. Let $M^{-1}=(n_{ij})_{1\le i,j\le n}$ be its inverse. The integer $d(f^{-1})$ is then equal to (\cite{gsp}, Remark 3.2)
% $$\max_{1\le i\le n} \Big(0,\sum_{j=1}^na_{ij}\Bigr) -\sum_{j=1}^n\min_{1\le i\le n}(0,n_{ij}).$$
% This seems difficult to estimate.

Another set of examples is provided by matrices of the form
$$
\begin{pmatrix}
1  & 1&   1&   d-3&   0&   \cdots&   \cdots&\cdots&\cdots&   0\\
0  & 1 &  1 &  d-2 &  0 &  \cdots &  \cdots&\cdots&\cdots &  0\\
d-1 &  0&   1&  0& 0 &  \cdots &  \cdots&\cdots&\cdots &  0\\
a_{30} &  a_{31}&  0 & a_{33} &0 &  \cdots &  \cdots&\cdots&\cdots &  0\\
a_{40}&a_{41}&a_{42}&a_{43}&1&0&\cdots&\cdots&\cdots&0\\
0&0&0&0&d-1&1&0&\cdots&\cdots&0\\
a_{60}&a_{61}&a_{62}&a_{63}&a_{64}&a_{65}&1&0&\cdots&0\\
\vdots&&&&&&\ddots&\ddots&\ddots&\vdots\\
\vdots&&&&&&&\ddots&\ddots&0\\
&&&&&&&&&1
\end{pmatrix}$$
where  $a_{30} +  a_{31}+ a_{33}=d\ge 3$, $a_{31}a_{33}\ne 0$, and, for each even $i$, we choose $a_{i0}a_{i2}\ne 0$ and $\sum_{j=0}^{i-1}a_{ij}=d-1$. 
\end{example}

\section{Degrees of a monomial map}
\label{sec:3}

Let $g:\P^n\dra \P^n$ be a    rational map. One defines the $i$th degree $d_i(g)$ as the degree of the image by $g$ of a general $\P^i\subset \P^n$ (more precisely, $d_i(g):=\P^{n-i}\cdot f_*\P^i$). One has $d_0(g)=1$, $d_n(g)=\deg(g)$, and $d_1(g)$ is the integer $d(g)$ defined earlier (\ie, the common degree of the components $  g_0,\dots,g_n $ of $g$, {\em provided they have no common factors).} An alternative definition of the $d_i(g)$ is as follows: if $\Gamma_g\subset \P^n\times\P^n$ is the graph of $g$,
\begin{equation}\label{gra}
d_i(g)=\Gamma_g \cdot p_1^*\P^i\cdot p_2^*\P^{n-i}.
\end{equation}

%If $f$ is dominant (\ie, if $d_n(f)\ne 0$), 
The sequence $d_0(g),\dots,d_n(g)$ is known to be a log-concave sequence: it satisfies
$$\forall i\in \{1,\dots,n-1\}\qquad d_i(g)^2\ge d_{i+1}(g)d_{i-1}(g)$$
(this is a direct consequence of the Hodge Index Theorem; \cite{dol}, (1.6)). This implies $d_i(g)\le d_1(g)^i$.

\begin{proposition}\label{bou}
 Let $f:\P^n \dra \P^n $ be   a dominant   map  defined by monomials of degree $d$ with no common factors.
 Set $c:= \lfloor n/2\rfloor+1$.
  If the degree of $f$ is not divisible by $d$, we have, for all $i\in\{c,\dots, n\}$,
 $$
 d_i(f)  \le  (1-d^{-c})^{\frac{i-1}{c-1}}d^i
 .$$
\end{proposition}

\begin{proof}
The degrees of $f$ can be expressed in terms of the Segre class of its indeterminacy locus $B$. In particular,   if $c':=\codim(B)$, one has (\cite{dol}, Proposition 2.3.1)
\begin{equation}\label{seg}
d_i(f)=\begin{cases}d^i\quad&\text{for }i<c',\\
d^i-\deg_s(B)\quad&\text{for }i=c',\end{cases}
\end{equation}
where $\deg_s(B)$ is the sum of the degrees of the top-dimensional components of $B$, counted with their {\em Samuel multiplicity} (this is larger than the `` usual'' multiplicity (\cite{F}, Examples 4.3.4 and 4.3.5.(c)); in particular $\deg_s(B)\ge \deg(B)$). Since in our case $c'\le c=n-\lceil (n-2)/2\rceil$ (Theorem \ref{main}), it follows from   log-concavity that we have  
$$d_c(f)< d^c.$$
By log-concavity, this implies that for $i\ge c$, one has
$$d_i(f)  \le d_c(f)^{\frac{i-1}{c-1}}d^{1-\frac{i-1}{c-1}}
\le (d^c-1)^{\frac{i-1}{c-1}}d^{ \frac{c-i}{c-1}}
=  (1-d^{-c})^{\frac{i-1}{c-1}}d^i.
$$
This proves the proposition.\qed
\end{proof}

When $g$  is birational, \ie, when $d_n(g)=1$, it   follows from (\ref{gra}) that  $d_i(g^{-1})=d_{n-i}(g)$ for all $i\in \{0,\dots,n\}$. In particular, 
$$d(g^{-1})=d_{n-1}(g)\le d(g)^{n-1}.$$
 By (\ref{seg}), equality occurs exactly when the indeterminacy locus of $g$ is finite.

 {\em When $f$ is a monomial birational transformation of $\P^n$,} Proposition \ref{bou}   gives the stronger bound: 
\begin{equation}\label{bound}
d(f^{-1}) \le (1-d^{-c})^{\frac{n-2}{c-1}}d^{n-1}=d^{n-1}-\frac{n-2}{\lfloor n/2\rfloor}d^{\lfloor (n-3)/2\rfloor}+O(d^{-2}),
\end{equation}
where $d:=d(f)$. However, as mentioned in the introduction, this  is not optimal. 

When $d(f)=2$,  the set of possible values for  $d(f^{-1})$ is $\{2,\dots,n\}$ and the maximal value $n$ is obtained only (up to permutation of the factors) for the birational map $f_{n,2}$ of (\ref{besta}) (\cite{cs}, Theorem 2.6). In particular, the other degrees of $f$ are then   fixed.

\medskip\noindent{\bf Johnson's  calculations.} When $d:=d(f)>2$, Johnson's computer calculations in \cite{joh} suggest that the maximal possible value for
$d(f^{-1})$ should be
$$d(f_{n,d}^{-1})=\frac{(d-1)^n-1}{d-2}=d^{n-1}-(n-2)d^{n-2}+O(d^{n-3})
$$
and that equality should only be attained when (up to permutation of the factors)  $f=f_{n,d}$. More precisely, Johnson checks that when $n=4$ and $3\le d \le 5$, one has $d(f^{-1})\le d(f_{n,d}^{-1})-d+1 $ if (up to permutation of the factors)  $f\ne f_{n,d}$. There are also further gaps in the list of possible values for $d(f^{-1})$.

\medskip\noindent{\bf Mixed volumes.} 
%As explained in \cite{huh}, 
The degrees $d_i(f)$ of a monomial map $f$ can be interpreted in terms of mixed volumes of polytopes in $\R^n$ as follows.
Let $\Delta\subset\R^n$ be the standard $n$-dimensional
simplex $\conv(0,{\mathbf e}_1,\dots,{\mathbf e}_n)$. Let $f:\P^n\dra \P^n$ be a monomial map with associated matrix $A=(a_{ij})_{0\le i,j\le n}$, and let $\Delta_f\subset \R^n$ be the simplex 
 which is the convex hull of the points ${\mathbf a}_i=(a_{i1},\dots,a_{in})\in \N^n$, for $i\in\{0,\dots,n\}$. Then (\cite{dol}, \S3.5)
 %(\cite{tv}, Corollary 2.5 or \cite{huh}, Theorem 6 and \S5.1.3) 
 $$d_i( f)=\MV(
\underbrace{\Delta,\ldots,\Delta}_{n-i\text{ times}},
\underbrace{\Delta_f,\ldots,\Delta_f}_{i\text{ times}}).
 $$
The right-hand side of this equality is a {\em mixed volume:} if the $n$-dimensional volume is normalized so that
$\vol (\Delta)=1/n!$, this is   $(n-i)!i!$ times the coefficient of $u^{n-i}v^i$ in the polynomial
 $\vol (u\Delta+v\Delta_f)$, where $u\Delta+v\Delta_f$ is the Minkowski sum $\{u{\mathbf x}+v{\mathbf y}\mid {\mathbf x}\in \Delta, {\mathbf y}\in\Delta_f \}$.
 
 Although mixed volumes are notoriously difficult to compute, there are computer programs  such as {\tt PHCpack} (available on Jan Verschelde's webpage)  that can do that. We should also mention the article \cite{alu}, which expresses the degrees of a monomial rational transformation in terms of integrals over an associated Newton
region.


\begin{thebibliography}{GSP}
 
 \bibitem[A]{alu} Aluffi, P., Multidegrees of  monomial rational maps, {\tt arXiv:1308.4152 [math.AG]}
 


  \bibitem[CS]{cs} Costa, B., Simis, A.,
Cremona maps defined by monomials, {\it  J. Pure Appl. Algebra}  {\bf 216} (2012), 202--215.
 

     \bibitem[D]{dol} Dolgachev, I., {\em  Lectures on Cremona transformations,} 2011, available at\\ {\tt http://www.math.lsa.umich.edu/$\sim$idolga/cremonalect.pdf}
     
      \bibitem[F]{F} Fulton, W.,  {\em Intersection Theory,} Springer Verlag, Berlin, 1984.

 
 \bibitem[GSP]{gsp} Gonzalez-Sprinberg, G., Pan, I., On the Monomial Birational Maps of the Projective Space, {\it Anais da Academia Brasileira de Ci\^encias}  {\bf 75} (2003), 129--134.
 
% \bibitem[H]{huh} Huh, J., 
%Milnor numbers of projective hypersurfaces and the chromatic polynomial of graphs, {\it J. Amer. Math. Soc.} {\bf 25} (2012), 907--927.
 
 \bibitem[J]{joh}  Johnson, P., Inverses of monomial Cremona transformations, {\tt arXiv:1105.1188 [math.AG]}
 
     \bibitem[SV]{sv1} Simis, A., Villarreal, R.,
Constraints for the normality of monomial subrings and birationality, {\it
Proc. Amer. Math. Soc.}  {\bf 131} (2003),  2043--2048.

%\bibitem[TV]{tv} Trung, N. V., Verma, J. K., Mixed multiplicities of ideals versus mixed volumes of polytopes,  {\it Trans. Amer. Math. Soc.}  {\bf 359} (2007),  4711--4727.
 
 
 \end{thebibliography}
\end{document}